\documentclass[11pt]{article}

\usepackage{amsmath,amssymb,amsthm}
\usepackage{mathrsfs}
\usepackage{hyperref}
\usepackage{tikz-cd}
\usepackage{graphicx} 
\theoremstyle{definition}
\newtheorem{definition}{Definition}[section]

\theoremstyle{plain}
\newtheorem{theorem}[definition]{Theorem}
\newtheorem{proposition}[definition]{Proposition}
\newtheorem{lemma}[definition]{Lemma}
\newtheorem{corollary}[definition]{Corollary}

\newtheorem{remark}[definition]{Remark}

\numberwithin{equation}{section}

\title{Hopf images of coactions and effective symmetry of quantum principal bundles}

\author{Arnab Bhattacharjee \thanks{Mathematical Institute of Charles University, email: arnabbhatta7@gmail.com\\ \\2020 \emph{Mathematics Subject Classification.} Primary: 16T20; Secondary: 58B32, 81R50\\ \emph{Key words and phrases}.
Hopf images, quantum principal bundles, inner--faithful coactions,
quantum groups, noncommutative geometry.\\ The author acknowledges support from HORIZON-MSCA-2021-SE-01-CaLIGOLA}}

\date{}

\begin{document}

\maketitle

\begin{abstract}
We introduce Hopf images of coactions of Hopf algebras and develop their role in
the geometry of quantum principal bundles.
Assuming cosemisimplicity of the structure Hopf algebra, we show that every
quantum principal bundle equipped with a right-covariant first-order
differential calculus admits a canonical and functorial reduction to one with
inner--faithful quantum symmetry.
This yields a classification of quantum principal bundles up to effective
quantum symmetry and a rigidity result identifying the minimal effective
symmetry acting on the reduced total space.
Examples from quantum groups are discussed.
\end{abstract}

\tableofcontents

\section{Introduction}

Symmetry is a fundamental organizing principle in geometry.
In classical differential geometry, symmetries of spaces are encoded by smooth
actions of Lie groups, and principal bundles provide the natural framework in
which global symmetries and local geometric data interact.
A well-known phenomenon in this setting is that a group action need not be
effective: a nontrivial normal subgroup may act trivially on the total space.
In such cases, one may pass canonically to a reduced principal bundle with
effective structure group by quotienting out the kernel of the action.

In noncommutative geometry, symmetry is typically encoded by coactions of Hopf
algebras or, in analytic settings, by actions of compact quantum groups \cite{banica2010quantum, bhowmick2009quantum, goswami2009quantum, Wang1998Quantum}.
Quantum principal bundles, introduced algebraically by Brzezi\'nski and Majid
\cite{brzezinski1993quantum} and further developed in various directions
\cite{brzezinski1999coalgebra, hajac1996strong, schneider1990principal}, provide a natural extension of
the classical notion of principal bundles to the noncommutative realm.
Despite the extensive development of the theory, comparatively little attention
has been paid to the question of \emph{effectiveness} of quantum symmetries:
given a Hopf algebra coaction on a quantum principal bundle, it is not a priori
clear whether the full Hopf algebra acts faithfully on the total space, or
whether part of the symmetry is redundant.

A related notion has appeared in the theory of quantum groups under the name of
\emph{inner--faithful actions} and \emph{Hopf images}.
Originally introduced by Banica and Bichon in the context of representations of
Hopf algebras \cite{banica2010hopf}, the Hopf image captures the smallest Hopf
quotient through which a given representation factors.
Dual formulations of this construction, emphasizing Hopf subalgebras rather than
quotients, have since appeared in various contexts
\cite{chirvasitu2019quantum, etingof2017hopf}.
These ideas suggest a natural approach to extracting the effective part of a
quantum symmetry directly from a given coaction.

The first goal of this paper is to develop a systematic theory of Hopf images for
\emph{coactions} of Hopf algebras on algebras.
Given a right coaction
\[
\delta : A \longrightarrow A \otimes H
\]
of a Hopf algebra $H$ on an algebra $A$, we define the \emph{Hopf image} of
$\delta$ as the smallest Hopf subalgebra $H_\delta \subseteq H$ through which the
coaction factors.
This Hopf image is characterized by a universal property and depends only on the
effective part of the coaction.
We call a coaction \emph{inner--faithful} if its Hopf image coincides with the
ambient Hopf algebra.
In this sense, inner--faithful coactions provide the appropriate notion of
effective quantum symmetry, extending the classical concept of effective group
actions.

The second and main objective of the paper is to apply this construction to
quantum principal bundles equipped with a right-covariant first-order
differential calculus.
Assuming that the structure Hopf algebra is cosemisimple, we show that every
quantum principal bundle admits a canonical reduction to a quantum principal
bundle whose structure Hopf algebra acts inner--faithfully.
This reduction is achieved by passing to a suitable quotient of the total space
algebra by a largest coaction-stable two-sided ideal, allowing the Hopf image
coaction to descend.
The resulting object, which we call the \emph{Hopf-image reduction}, may be
viewed as the noncommutative analogue of reducing a classical principal bundle
to one with effective structure group.

Beyond existence, we prove that Hopf-image reduction is \emph{functorial}.
More precisely, we introduce a category of quantum principal bundles with
cosemisimple structure Hopf algebras and show that Hopf-image reduction defines a
functor to the full subcategory consisting of bundles with inner--faithful
coactions.
This leads naturally to an equivalence relation on quantum principal bundles,
interpreted as equivalence up to effective quantum symmetry, and reduces the
classification problem to that of inner--faithful quantum principal bundles.

Finally, we establish a rigidity result showing that Hopf-image reduction yields
the \emph{minimal} effective quantum symmetry acting on the reduced total space:
any other inner--faithful quantum symmetry acting on the same reduced bundle must
contain the Hopf image as a Hopf subalgebra.
This result shows that Hopf-image reduction is not merely canonical, but rigid in
a precise sense.

We illustrate the general theory with examples of inner--faithful coactions,
including the coproduct of Hopf algebras, as well as
coactions arising from quantized coordinate algebras of complex semisimple Lie
groups.
These examples demonstrate that effective quantum symmetry arises naturally in
quantum group geometry and highlight the relevance of Hopf images in concrete
geometric situations.

\medskip
\noindent
\textbf{Structure of the paper.}
In Section~2 we recall basic notions concerning Hopf algebras, coactions,
first-order differential calculi, and quantum principal bundles.
Section~3 introduces the Hopf image of a coaction, establishes its basic
properties, and develops the notion of inner--faithful coactions as a precise
formulation of effective quantum symmetry.
In Section~4 we construct the Hopf-image reduction of a quantum principal
bundle and prove that, under the assumption of cosemisimplicity, this reduction
yields a new quantum principal bundle with inner--faithful structure Hopf
algebra.
Section~5 is devoted to functoriality and classification: we define the
appropriate categories of quantum principal bundles, show that Hopf-image
reduction is functorial and essentially surjective, and formulate a
classification result up to effective quantum symmetry.
Finally, in Section~6 we present examples of inner--faithful coactions,
including those arising from quantized coordinate algebras of complex
semisimple Lie groups and from canonical Hopf algebra coactions.

\medskip
\noindent

\section{Preliminaries}

\subsection{Hopf algebras and coactions}

Throughout the paper, $H$ denotes a Hopf algebra with coproduct $\Delta$,
counit $\varepsilon$, and a bijective antipode $S$.
An algebra $A$ is always assumed to be associative and unital. In this section we follow literature from text books \cite{beggs2020quantum, brzezinski2003corings, klimyk2012quantum}.

\begin{definition}
A \emph{right coaction} of a Hopf algebra $H$ on an algebra $A$ is a linear map
\[
\delta : A \longrightarrow A \otimes H
\]
satisfying
\[
(\delta \otimes \mathrm{id}_{H})\circ \delta
= (\mathrm{id}_{A} \otimes \Delta)\circ \delta,
\qquad
(\mathrm{id}_{A} \otimes \varepsilon)\circ \delta = \mathrm{id}_{A}.
\]
\end{definition}

We use Sweedler notation $\delta(a)= a_{(0)}\otimes a_{(1)}$.
\subsection{First order differential calculus}

A \emph{first-order differential calculus} (fodc) over an algebra $A$ is a pair $(\Omega^1(A), \mathrm{d})$, where $\Omega^1(A)$ is a $A$-bimodule and $\mathrm{d}: A \to \Omega^1(A)$ is a derivation such that $\Omega^1(A)$ is generated as a left $A$-module by those elements of the form $\mathrm{d}a$, for $a \in A$.
\subsection{Quantum principal bundles}

Let $H$ be a Hopf algebra and $(A, \delta)$ be a right $H$-comodule algebra. A fodc over $A$ is right $H$-covariant if there exist (well defined) a morphism
\[
\Delta_{R}: \Omega^{1}(A)\rightarrow \Omega^{1}(A)\otimes H, \quad \Delta_{R}(a\mathrm{d}b)= a_{(0)}\mathrm{d}b_{(0)}\otimes a_{(1)}b_{(1)}.
\]
\begin{definition}
   A \emph{quantum principal $H$-bundle} is a pair $(A,\Omega^{1}(A))$ consisting of right $H$-comodule algebra $(A, \delta)$ and right $H$-covariant fodc $\Omega^{1}(A)$ such that
   \begin{enumerate}
       \item $A$ is a Hopf-Galois extension of $B= A^{co(H)}=\lbrace a\in A : \delta(a)= a\otimes 1_{H}\rbrace$ this means
       $$\mathrm{can}:= (\cdot\otimes \mathrm{id})\circ (\mathrm{id}\otimes \Delta_{R}): A\otimes_{B}A\rightarrow A\otimes H$$ is a bijection.
       \item If $N\subset \Omega^{1}_{u}(A)$ is the $A$-sub-bimodule of the universal calculus corresponding to $\Omega^{1}(A)$, we have $\mathrm{ver}(N)= A\otimes I$, for some $\mathrm{Ad}_{R}$-subcomodule right ideal $I\subseteq H^{+}$, where $\mathrm{Ad}_{R}: H\rightarrow H\otimes H$ and is defined as $\mathrm{Ad}_{R}(h)= h_{(2)}\otimes S(h_{(1)})h_{(3)}$ and $H^{+}= \mathrm{Ker}(\varepsilon_{H})$.
   \end{enumerate}
\end{definition}

In this paper, the emphasis is placed on the coaction $\delta$ as the carrier of
quantum symmetry.

\section{Hopf image of a coaction}
The notion of Hopf image, introduced by Banica and Bichon \cite{banica2010hopf} for representations of
Hopf algebras, provides a canonical way to isolate the effective quantum symmetry
of an action.
In the present section, we develop a dual formulation of this concept for coactions
of Hopf algebras on algebras.
Replacing Hopf quotients by Hopf subalgebras, we define the Hopf image of a
coaction as a universal factorization object encoding its effective symmetry.
This construction plays a central role in the reduction and classification of
quantum principal bundles considered in later sections.
 
\subsection{Factorizations of coactions}

\begin{definition}\label{def:facto}
Let $\delta:A\to A\otimes H$ be a right coaction of a Hopf algebra $H$ on an algebra $A$. A \emph{factorization} of $\delta$ is a triple $(L,\iota_L,\delta_L)$ where
$L\subseteq H$ is a Hopf subalgebra with inclusion $\iota_L:L\hookrightarrow H$
and $\delta_L:A\to A\otimes L$ is a right $L$-coaction such that
\[
\delta=(\mathrm{id}\otimes\iota_L)\circ\delta_L.
\]
\end{definition}

\subsection{Hopf image of a coaction}

\begin{definition}\label{def:Hopfimage}
Let $\delta:A\to A\otimes H$ be a right coaction of a Hopf algebra $H$ on an algebra $A$. The \emph{Hopf image} of $\delta$, denoted $H_\delta$, is the initial object in
the category of factorizations of $\delta$.
\end{definition}

\begin{theorem}\label{thm:Hopfimage}
Let $H$ be a Hopf algebra, $A$ an algebra, and
\[
\delta : A \longrightarrow A \otimes H
\]
a right coaction. Then there exists a Hopf subalgebra
\[
H_\delta \subseteq H
\]
and a right $H_\delta$-coaction
\[
\delta_{\mathrm{im}} : A \longrightarrow A \otimes H_\delta
\]
such that:
\begin{enumerate}
\item $\delta = (\mathrm{id}_A \otimes \iota_{H_{\delta}})\circ \delta_{\mathrm{im}}$, where
      $\iota_{H_{\delta}} : H_\delta \hookrightarrow H$ is the inclusion;
\item $(H_\delta,\iota_{H_{\delta}},\delta_{\mathrm{im}})$ is the initial object in the category of factorizations of $\delta$.
\end{enumerate}
Equivalently, $H_\delta$ is the smallest Hopf subalgebra $L \subseteq H$ such that
\[
\delta(A) \subseteq A \otimes L .
\]
\end{theorem}
\begin{proof}
Let $\mathcal{S}$ denote the collection of all Hopf subalgebras
$L \subseteq H$ satisfying $\delta(A) \subseteq A \otimes L$.
Define
\[
H_\delta := \bigcap_{L \in \mathcal{S}} L .
\]

Since intersections of subalgebras (resp.\ subcoalgebras) are again
subalgebras (resp.\ subcoalgebras), and the antipode preserves
intersections, $H_\delta$ is a Hopf subalgebra of $H$.

By construction, $\delta(A) \subseteq A \otimes L$ for every
$L \in \mathcal{S}$, hence
\[
\delta(A) \subseteq A \otimes \bigcap_{L \in \mathcal{S}} L
= A \otimes H_\delta .
\]
Thus $\delta$ uniquely factors as
\[
\delta = (\mathrm{id}_A \otimes \iota_{H_{\delta}})\circ \delta_{\mathrm{im}},
\]
where $\delta_{\mathrm{im}}$ is the restriction of $\delta$ to
$A \otimes H_\delta$.

Now let $(L,\iota_{L},\delta_L)$ be any factorization of $\delta$.
By definition, $L \in \mathcal{S}$, hence $H_\delta \subseteq L$.
The inclusion $f : H_\delta \hookrightarrow L$ is the unique Hopf algebra
morphism satisfying
\[
\delta_L = (\mathrm{id}_A \otimes f)\circ \delta_{\mathrm{im}}
\quad\text{and}\quad
(\mathrm{id}_{A}\otimes \iota_{L})\circ (\mathrm{id}_{A}\otimes f)= \mathrm{id}_{A}\otimes \iota_{H_{\delta}}
\]
Therefore $(H_\delta,\iota_{H_{\delta}},\delta_{\mathrm{im}})$ is initial among
all factorizations of $\delta$.
\end{proof}

\begin{proposition}\label{prop:generator}
Let $H$ be a Hopf algebra, $A$ an algebra, and
\[
\delta : A \longrightarrow A \otimes H
\]
a right $H$-coaction. Then the Hopf image $H_\delta$ of $\delta$ coincides with
the Hopf subalgebra of $H$ generated by the set
\[
\mathcal{C}_\delta
=
\{\, (\omega \otimes \mathrm{id})\delta(a)
\mid a \in A,\ \omega \in A^* \,\}.
\]
\end{proposition}

\begin{proof}
Let $K \subseteq H$ be the Hopf subalgebra generated by the set
$\mathcal{C}_\delta$.
For each $a \in A$, write $\delta(a)=\sum a_{(0)}\otimes a_{(1)}$.
Then $a_{(1)}\in\mathcal{C}_\delta\subseteq K$, hence
$\delta(a)\in A\otimes K$ for all $a\in A$.
Thus $\delta$ factors through $K$, and therefore $H_\delta\subseteq K$.

Since $\delta(A)\subseteq A\otimes H_\delta$, applying
$(\omega\otimes\mathrm{id})$ shows that
$\mathcal{C}_\delta\subseteq H_\delta$.
Because $H_\delta$ is a Hopf subalgebra, it contains the Hopf subalgebra
generated by $\mathcal{C}_\delta$, namely $K$.
Therefore $K\subseteq H_\delta$.
Combining the two inclusions yields $H_\delta=K$.
\end{proof}

\begin{lemma}\label{lem:minimality}
Let $H$ be a Hopf algebra, $A$ an algebra, and
\[
\delta : A \longrightarrow A \otimes H
\]
a right $H$-coaction. Let $(L,\iota_L,\delta_L)$ be any factorization of $\delta$, that is,
$L \subseteq H$ is a Hopf subalgebra with inclusion $\iota_L : L \hookrightarrow H$ and
$\delta_L : A \to A \otimes L$ is a right $L$-coaction such that
\[
\delta = (\mathrm{id}_A \otimes \iota_L)\circ \delta_L .
\]
Then the Hopf image $H_\delta$ of $\delta$ satisfies
\[
H_\delta \subseteq L .
\]
\end{lemma}

\begin{proof}
By definition, the Hopf image $H_\delta$ is the smallest Hopf subalgebra
$K \subseteq H$ such that $\delta(A) \subseteq A \otimes K$.
Since $(L,\iota_L,\delta_L)$ is a factorization of $\delta$, we have
$\delta(A) \subseteq A \otimes L$, hence $L$ is one of the Hopf subalgebras
appearing in the defining intersection for $H_\delta$.
Therefore $H_\delta \subseteq L$, as claimed.
\end{proof}

\begin{remark}
Our definition of the Hopf image of a coaction uses the largest possible class
of factorizations, namely all Hopf subalgebras $L \subseteq H$ through which the
coaction factors. As a consequence, the Hopf image depends only on the effective
part of the coaction and is independent of any chosen ambient Hopf algebra
containing it.
\end{remark}

\subsection{Inner-faithful coactions}

\begin{definition}\label{def:innerfaithful}
Let $H$ be a Hopf algebra, $A$ an algebra, and
\[
\delta : A \longrightarrow A \otimes H
\]
a right $H$-coaction. Let $H_\delta \subseteq H$ denote the Hopf image of $\delta$.
The coaction $\delta$ is called \emph{inner faithful} if
\[
H_\delta = H .
\]
\end{definition}

\begin{proposition}\label{prop:innerfaithful}
Let $H$ be a Hopf algebra, $A$ an algebra, and
\[
\delta : A \longrightarrow A \otimes H
\]
a right $H$-coaction. Let $H_\delta \subseteq H$ denote the Hopf image of $\delta$.
Then the following statements are equivalent:
\begin{enumerate}
\item $\delta$ is inner faithful, i.e.\ $H_\delta = H$;
\item There exists no proper Hopf subalgebra $L \subsetneq H$ such that
      $\delta(A) \subseteq A \otimes L$;
\item For every factorization $(L,\iota_L,\delta_L)$ of $\delta$, the inclusion
      $\iota_L : L \hookrightarrow H$ is an isomorphism.
\end{enumerate}
\end{proposition}

\begin{proof}
$(1) \Rightarrow (2)$.
If $\delta(A) \subseteq A \otimes L$ for some Hopf subalgebra $L \subseteq H$,
then $(L,\iota_L,\delta_L)$ is a factorization of $\delta$.
By minimality of the Hopf image, $H_\delta \subseteq L$.
If $L$ were proper, this would contradict $H_\delta = H$.
Hence no such proper Hopf subalgebra exists.

$(2) \Rightarrow (3)$.
Let $(L,\iota_L,\delta_L)$ be any factorization of $\delta$.
Then $\delta(A) \subseteq A \otimes L$, so by assumption $L = H$.
Therefore $\iota_L$ is an isomorphism.

$(3) \Rightarrow (1)$.
By definition, $(H_\delta,\iota_{H_{\delta}},\delta_{\mathrm{im}})$ is a factorization
of $\delta$. Applying $(3)$ to this factorization shows that
$\iota_{H_{\delta}} : H_\delta \hookrightarrow H$ is an isomorphism, hence
$H_\delta = H$.
\end{proof}

\begin{proposition}\label{prop:universitality}
Let $H$ be a Hopf algebra, $A$ an algebra, and
\[
\delta : A \longrightarrow A \otimes H
\]
a right $H$-coaction. Let $(H_\delta,\iota_{H_\delta},\delta_{\mathrm{im}})$ denote
the Hopf image of $\delta$, where
$\iota_{H_\delta} : H_\delta \hookrightarrow H$ is the inclusion and
$\delta_{\mathrm{im}} : A \to A \otimes H_\delta$ is the induced coaction. Then, the coaction $\delta_{\mathrm{im}}$ is inner faithful, and conversely, if $(L,\iota_L,\delta_L)$ is a factorization of $\delta$ such that
      the coaction $\delta_L : A \to A \otimes L$ is inner faithful, then
      $L \cong H_\delta$ as Hopf algebras.
\end{proposition}

\begin{proof}
By definition, $(H_\delta,\iota_{H_\delta},\delta_{\mathrm{im}})$ is a factorization
of $\delta$ and $H_\delta$ is the smallest Hopf subalgebra of $H$ such that
$\delta(A) \subseteq A \otimes H_\delta$.
Suppose that $\delta_{\mathrm{im}}$ factors through a proper Hopf subalgebra
$K \subsetneq H_\delta$, i.e.\ $\delta_{\mathrm{im}}(A) \subseteq A \otimes K$.
Then $(K,\iota_K,\delta_K)$ would be a factorization of $\delta$, contradicting
the minimality of $H_\delta$. Hence $\delta_{\mathrm{im}}$ is inner faithful.

Conversely, let $(L,\iota_L,\delta_L)$ be a factorization of $\delta$ such that $\delta_L$
is inner faithful. By minimality of the Hopf image, we have
$H_\delta \subseteq L$.
On the other hand, since $\delta = (\mathrm{id}\otimes\iota_L)\circ\delta_L$,
the image of $\delta_L$ is contained in $A \otimes L$, and by definition of the
Hopf image we also have $L_\delta = H_\delta$ for the coaction $\delta_L$.
Inner faithfulness of $\delta_L$ implies $L = L_\delta = H_\delta$.
Therefore $L \cong H_\delta$ as Hopf algebras.
\end{proof}
\begin{remark}
The Hopf image thus provides the unique factorization of a coaction with
inner-faithful structure Hopf algebra. In particular, every coaction admits a
canonical reduction to an inner-faithful one by passing to its Hopf image.
\end{remark}
\begin{proposition}
Let $H$ be a Hopf algebra and
\[
\delta_A : A \longrightarrow A \otimes H
\]
a right $H$-coaction on an algebra $A$. Let $\theta : A \to B$ be an algebra
isomorphism, and define a right $H$-coaction on $B$ by
\[
\delta_B := (\theta \otimes \mathrm{id}) \circ \delta_A \circ \theta^{-1}
: B \longrightarrow B \otimes H .
\]
Then the Hopf images of $\delta_A$ and $\delta_B$ are isomorphic as Hopf
algebras. More precisely,
\[
H_{\delta_A} = H_{\delta_B} \subseteq H .
\]
\end{proposition}
\begin{proof}
Recall that the Hopf image of a coaction is the smallest Hopf subalgebra of $H$
through which the coaction factors. Equivalently, it is the Hopf subalgebra
generated by the coefficients of the coaction.

Let $a \in A$ and write
\[
\delta_A(a) =  a_{(0)} \otimes a_{(1)} .
\]
For $b = \theta(a) \in B$, we have
\[
\delta_B(b)
= (\theta \otimes \mathrm{id})\delta_A(a)
=  \theta(a_{(0)}) \otimes a_{(1)} .
\]
Thus the coefficients appearing in $\delta_B(b)$ are exactly the same elements
$a_{(1)} \in H$ as those appearing in $\delta_A(a)$.

Since $\theta$ is an algebra isomorphism, every element of $B$ is of the form
$\theta(a)$ for some $a \in A$, and therefore the set of coefficients of
$\delta_B$ coincides with the set of coefficients of $\delta_A$.
Consequently, the Hopf subalgebra of $H$ generated by the coefficients of
$\delta_B$ is equal to the Hopf subalgebra generated by the coefficients of
$\delta_A$.

By the coefficient description of the Hopf image, it follows that
\[
H_{\delta_B} = H_{\delta_A},
\]
which proves the claim.
\end{proof}
\begin{remark}
The Hopf image of a coaction is therefore invariant under algebraic
isomorphisms of the underlying algebra. In particular, the notion of effective
quantum symmetry encoded by the Hopf image depends only on the coaction up to
algebra isomorphism.
\end{remark}

\begin{definition}\label{def:effective symmetry}
Let $\delta : A \to A \otimes H$ be a right coaction of a Hopf algebra $H$ on an
algebra $A$.
The \emph{effective quantum symmetry} of $\delta$ is a Hopf algebra $L$ together
with a right $L$-coaction
\[
\delta_L : A \to A \otimes L
\]
such that:
\begin{enumerate}
\item[(i)] $\delta_L$ is inner faithful;
\item[(ii)] There exists a Hopf algebra morphism $\iota : L \hookrightarrow H$
for which
\[
\delta = (\mathrm{id} \otimes \iota) \circ \delta_L;
\]
\item[(iii)] $(L,\delta_L)$ is \emph{universal} with respect to properties
\emph{(i)} and \emph{(ii)}, in the sense that for any Hopf algebra $L'$ and any
inner-faithful coaction $\delta_{L'} : A \to A \otimes L'$ factoring $\delta$,
there exists a unique Hopf algebra morphism
$L \to L'$ intertwining the coactions.
\end{enumerate}
\end{definition}

\begin{proposition}
Let $\delta : A \to A \otimes H$ be a right coaction.
Then the Hopf image $(H_\delta,\delta_{\mathrm{im}})$ of $\delta$ realizes the
effective quantum symmetry of $\delta$ in the sense of the definition ~\ref{def:effective symmetry}.
In particular, the effective quantum symmetry is unique up to unique Hopf
algebra isomorphism and coincides with $H_\delta$.
\end{proposition}

\begin{proof}
By construction, the Hopf image $H_\delta$ is a Hopf subalgebra of $H$ and the
induced coaction $\delta_{\mathrm{im}} : A \to A \otimes H_\delta$ satisfies
$\delta = (\mathrm{id} \otimes \iota_{H_{\delta}})\circ\delta_{\mathrm{im}}$.

By minimality of the Hopf image, $\delta_{\mathrm{im}}$ does not factor through
any proper Hopf subalgebra of $H_\delta$, hence it is inner faithful.
Moreover, if $(L,\delta_L)$ is any inner-faithful factorization of $\delta$,
then $H_\delta \subseteq L$ by the universal property of the Hopf image.
Thus $(H_\delta,\delta_{\mathrm{im}})$ is universal among inner-faithful
factorizations of $\delta$, and therefore realizes the effective quantum
symmetry.
\end{proof}

\begin{proposition}
    Let $H_{1}, H_{2}$ be Hopf algebras and $A,B$ be right $H_1$ and $H_2$ comodule algebras respectively. Let $\delta_{1}: A\rightarrow A\otimes H_1$ and $\delta_{2}: B\rightarrow B\otimes H_{2}$ be right coactions respectively. Then
    \begin{enumerate}
        \item $\delta_{3}:= (\mathrm{id}_{A}\otimes \tau\otimes \mathrm{id}_{H_{2}})\circ (\delta_{1}\otimes \delta_{2})$ is a right coaction of $H_{1}\otimes H_{2}$ on $A\otimes B$, where $\tau: H_{1}\otimes B\rightarrow B\otimes H_{1}$ is a flip map;

        \item There exists an injective Hopf algebra map from $(H_{1}\otimes H_{2})_{\delta_{3}}$ to $H_{1,\delta_{1}}\otimes H_{2, \delta_{2}}$.
    \end{enumerate}
\end{proposition}

\begin{proof}
    \begin{enumerate}
        \item In order to show that $\delta_{3}$ is a right coaction of $H_{1}\otimes H_{2}$ on $A\otimes B$, it is sufficient to show the following identities:
        \begin{enumerate}
            \item $(\mathrm{id}_{A\otimes B}\otimes \Delta_{H_{1}\otimes H_{2}})\circ \delta_{3}= (\delta_{3}\otimes \mathrm{id}_{H_{1}\otimes H_{2}})\circ\delta_{3}$;
            \item $(\mathrm{id}_{A\otimes B}\otimes \epsilon_{H_{1}\otimes H_{2}})\circ \delta_{3}= \mathrm{id}_{A\otimes B}$
        \end{enumerate}
        Where, $\Delta_{H_{1}\otimes H_{2}}$ and $\epsilon_{H_{1}\otimes H_{2}}$ are coproduct and antipode repectively on $H_{1}\otimes H_{2}$. Moreover, $\Delta_{H_{1}\otimes H_{2}}$ is given by $(\mathrm{id}_{H_{1}}\otimes T\otimes \mathrm{id}_{H_{2}})\circ (\Delta_{H_{1}}\otimes \Delta_{H_{2}})$, where $T$ is a flip map between $H_{1}\otimes H_{2}$ and $H_{2}\otimes H_{1}$, and $\epsilon_{H_{1}\otimes H_{2}}$ is given by $\epsilon_{H_{1}}\otimes \epsilon_{H_{2}}$.\\
        \\
        We consider $a\otimes b\in A\otimes B$, and we have the following,
        \begin{align*}
            (\mathrm{id}_{A\otimes B}\otimes \Delta_{H_{1}\otimes H_{2}})\circ \delta_{3}(a\otimes b)&= (\mathrm{id}_{A\otimes B}\otimes \Delta_{H_{1}\otimes H_{2}})(\mathrm{id}_{A}\otimes \tau\otimes \mathrm{id}_{H_{2}})\circ (\delta_{1}\otimes \delta_{2})(a\otimes b)\\
            &= (\mathrm{id}_{A\otimes B}\otimes \Delta_{H_{1}\otimes H_{2}})(\mathrm{id}_{A}\otimes \tau\otimes \mathrm{id}_{H_{2}})\circ (\delta_{1}a\otimes \delta_{2}b)\\
            &= (\mathrm{id}_{A\otimes B}\otimes \Delta_{H_{1}\otimes H_{2}})(\mathrm{id}_{A}\otimes \tau\otimes \mathrm{id}_{H_{2}})(a_{(0)}\otimes a_{(1)}\otimes b_{(0)}\otimes b_{(1)})\\
            &= (\mathrm{id}_{A\otimes B}\otimes \Delta_{H_{1}\otimes H_{2}})(a_{(0)}\otimes b_{(0)}\otimes a_{(1)}\otimes b_{(1)})\\
            &= a_{(0)}\otimes b_{(0)}\otimes \Delta_{H_{1}\otimes H_{2}}(a_{(1)}\otimes b_{(1)})\\
            &= a_{(0)}\otimes b_{(0)}\otimes a_{(1)}\otimes b_{(1)}\otimes a_{(2)}\otimes b_{(2)}
        \end{align*}
        Next, \begin{align*}
            (\delta_{3}\otimes \mathrm{id}_{H_{1}\otimes H_{2}})\circ \delta_{3}(a\otimes b)&=  (\delta_{3}\otimes \mathrm{id}_{H_{1}\otimes H_{2}})(\mathrm{id}_{A}\otimes \tau\otimes \mathrm{id}_{H_{2}})\circ (\delta_{1}\otimes \delta_{2})(a\otimes b)\\
            &= (\delta_{3}\otimes \mathrm{id}_{H_{1}\otimes H_{2}})(a_{(0)}\otimes b_{(0)}\otimes a_{(1)}\otimes b_{(1)})\\
            &= \delta_{3}(a_{(0)}\otimes b_{(0)})\otimes a_{(1)}\otimes b_{(1)}\\
            &= a_{(0)}\otimes b_{(0)}\otimes a_{(1)}\otimes b_{(1)}\otimes a_{(2)}\otimes b_{(2)}
        \end{align*}
    
    Therefore, we have the identity $(a)$ i.e. $(\mathrm{id}_{A\otimes B}\otimes \Delta_{H_{1}\otimes H_{2}})\circ \delta_{3}= (\delta_{3}\otimes \mathrm{id}_{H_{1}\otimes H_{2}})\circ\delta_{3}$.
    Next, to show the identity $(b)$, we have the following
    \begin{align*}
        (\mathrm{id}_{A\otimes B}\otimes \epsilon_{H_{1}\otimes H_{2}})\circ \delta_{3}(a\otimes b)&= (\mathrm{id}_{A\otimes B}\otimes \epsilon_{H_{1}\otimes H_{2}})(\mathrm{id}_{A}\otimes \tau\otimes \mathrm{id}_{H_{2}})\circ (\delta_{1}\otimes \delta_{2})(a\otimes b)\\
        &= \mathrm{id}_{A\otimes B}\otimes \epsilon_{H_{1}\otimes H_{2}})(a_{(0)}\otimes b_{(0)}\otimes a_{(1)}\otimes b_{(1)})\\
        &= a_{(0)}\otimes b_{(0)}\otimes \epsilon_{H_{1}\otimes H_{2}}(a_{(1)}\otimes b_{(1)})\\
        &= a_{(0)}\otimes b_{(0)}\otimes \epsilon_{H_{1}}(a_{(1)})\otimes \epsilon_{H_{2}}(b_{(1)})\\
        &= a_{(0)}\otimes b_{(0)}\\
        &= a\otimes b
    \end{align*}
    \item 
We have $(H_{1,\delta_{1}}, \iota_{H_{1,\delta_{1}}}, \delta_{1,\mathrm{im}})$ and $(H_{2,\delta_{2}}, \iota_{H_{2,\delta_{2}}}, \delta_{2,\mathrm{im}})$ be Hopf images of the right coactions $\delta_{1}$ and $\delta_{2}$ respectively. Therefore, we have the following diagrams

\begin{tikzcd}
A \arrow[r, "\delta_{1}"] \arrow[rd, "{\delta_{1,\mathrm{im}}}"'] & A\otimes H_{1}                                                                               &  & B \arrow[rd, "{\delta_{2,\mathrm{im}}}"'] \arrow[r, "\delta_{2}"]                                                                                                     & B\otimes H_{2}                                                                              \\
                                                                  & {A\otimes H_{1,\delta_{1}}} \arrow[u, "{\mathrm{id}_{A}\otimes \iota_{H_{1,\delta_{1}}} }"'] &  &                                                                                                                                                                       & {B\otimes H_{2,\delta_{2}}} \arrow[u, "{\mathrm{id}_{B}\otimes \iota_{H_{2,\delta_{2}}}}"'] \\
                                                                  & A\otimes B \arrow[rr, "\delta_{3}"] \arrow[rrd, "\overline{\delta_{3}}"']                    &  & A\otimes B\otimes H_{1}\otimes H_{2}                                                                                                                                  &                                                                                             \\
                                                                  &                                                                                              &  & {A\otimes B\otimes H_{1,\delta_{1}}\otimes H_{2,\delta_{2}}} \arrow[u, "{\mathrm{id}_{A\otimes B}\otimes \iota_{H_{1,\delta_{1}}}\otimes \iota_{H_{2,\delta_{2}}}}"'] &                                                                                            
\end{tikzcd}
Here, $\overline{\delta_{3}}:= (\mathrm{id}_{A}\otimes \tau_{\mathrm{im}}\otimes \mathrm{id}_{H_{2}, \delta_{2}})\circ (\delta_{1,\mathrm{im}}\otimes \delta_{2,\mathrm{im}})$ and $\tau_{\mathrm{im}}: H_{1,\delta_{1}}\otimes B\rightarrow B\otimes H_{1, \delta_{1}}$ is a flip map. Since, by $(1)$, $\delta_{3}$ is a right coaction of $H_{1}\otimes H_{2}$, so, $(H_{1, \delta_{1}}\otimes H_{2, \delta_{2}}, \iota_{H_{1,\delta_{1}}}\otimes \iota_{H_{2,\delta_{2}}}, \overline{\delta_{3}})$ is a factorization of the coaction $\delta_{3}$. Therefore, by minimality of the Hopf image of $H_{1}\otimes H_{2}$, we have $(H_{1}\otimes H_{2})_{\delta_{3}}\subseteq H_{1,\delta_{1}}\otimes H_{2,\delta_{2}}$. Hence, there exists an injective Hopf algebra morphism from $(H_{1}\otimes H_{2})_{\delta_{3}}$ to $H_{1,\delta_{1}}\otimes H_{2,\delta_{2}}$.
    \end{enumerate}
\end{proof}
\section{Hopf image reduction of quantum principal bundles}
In this section we introduce the \emph{Hopf image reduction} of a quantum
principal bundle.
Given a quantum principal bundle $(A,\Omega^1(A),H,\delta)$, we associate to the
coaction $\delta$ its Hopf image and construct a canonically reduced quantum
principal bundle whose structure Hopf algebra acts inner--faithfully.
Assuming cosemisimplicity of $H$, we show that this reduction preserves the
principal bundle structure and is functorial with respect to morphisms.
We further prove a rigidity result showing that the Hopf image reduction is
minimal among all reductions with inner--faithful coaction.
\begin{lemma}\label{lem:idealcollection}
Let $\delta: A\rightarrow A\otimes H$ be a right coaction of a Hopf algebra $H$ on an algebra $A$ and $(H_{\delta}, \iota_{H_{\delta}}, \delta_{\mathrm{im}})$ be the Hopf image of the right coaction $\delta$.
Let $\mathcal C$ be the collection of two-sided ideals $J\subseteq A$ such that
\[
\delta_{\mathrm{im}}(J)\subseteq J\otimes H_{\delta} 
\]
Then $\mathcal C$ is closed under arbitrary sums:  
if $\{J_\alpha\}_\alpha\subseteq\mathcal C$ and $J:=\sum_\alpha J_\alpha$, then
\[
\delta_{\mathrm{im}}(J)\subseteq J\otimes H_{\delta} .
\]
In particular, the sum of all ideals in $\mathcal C$ again belongs to $\mathcal C$.
\end{lemma}

\begin{proof}
Since each $J_\alpha$ is a two-sided ideal, their sum $J$ is also a 
two-sided ideal.  
For every $\alpha$ we have 
\[
\delta_{\mathrm{im}}(J_\alpha)\subseteq J_\alpha\otimes H_{\delta} \subseteq J\otimes H_{\delta}.
\]
By linearity of $\delta_{\mathrm{im}}$,
\[
\delta_{\mathrm{im}}(J)=\sum_\alpha \delta_{H_{\delta}}(J_\alpha)
      \subseteq \sum_\alpha J_\alpha\otimes H_{\delta}
      = J\otimes H_{\delta}.
\]
\end{proof}

\begin{lemma}\label{lem:largestideal}
 Let $\delta: A\rightarrow A\otimes H$ be a right coaction of a Hopf algebra $H$ on an algebra $A$, and $(H_{\delta}, \iota_{H_{\delta}}, \delta_{\mathrm{im}})$ be the Hopf image of the right coaction $\delta$. Then there exists a largest two-sided ideal $I\subseteq A$ such that
\[
\delta_{\mathrm{im}}(I)\subseteq I\otimes H_{\delta}.
\]

\end{lemma}

\begin{proof}
Let $\mathcal C$ be the set of all two-sided ideals $J\subseteq A$ satisfying
$\delta_{\mathrm{im}}(J)\subseteq J\otimes H_{\delta}$.  
Define
\[
I:=\sum_{J\in\mathcal C} J.
\]
By Lemma~\ref{lem:idealcollection}, $I$ is a two-sided ideal and $\delta_{\mathrm{im}}(I)\subseteq I\otimes H_{\delta}$.
It is clearly the largest such ideal by construction.

  Since $\delta_{\mathrm{im}}(I)\subseteq I\otimes H_{\delta}$, and the map
\[
\mathrm{id}_I\otimes\iota_{H_{\delta}} : I\otimes H_{\delta} \to I\otimes H
\]
implies, $(\mathrm{id}_I\otimes\iota_{H_{\delta}})(I\otimes H_{\delta})= I\otimes H_{\delta}$. Therefore, $\delta(I)= (\mathrm{id}_I\otimes\iota_{H_{\delta}})\circ \delta_{\mathrm{im}}(I)\subseteq I \otimes H_{\delta}$.
\end{proof}

\begin{lemma}\label{lem:coaction-descends}
Let $A$ be an algebra and let $H_\delta$ be a Hopf algebra.
Suppose that
\[
\delta_{\mathrm{im}}: A \longrightarrow A\otimes H_\delta
\]
is a right $H_\delta$--coaction.
Let $I\subseteq A$ be a largest two-sided ideal defined in Lemma~ \ref{lem:largestideal} such that
\[
\delta_{\mathrm{im}}(I)\subseteq I\otimes H_\delta.
\]
Set $A_0:=A/I$ and denote by $\pi:A\to A_0$ the canonical quotient map.
Then there exists a unique algebra homomorphism
\[
\overline{\delta}_{\mathrm{im}}: A_0 \longrightarrow A_0\otimes H_\delta
\]
such that
\[
\overline{\delta}_{\mathrm{im}}\circ \pi
=
(\pi\otimes \mathrm{id})\circ \delta_{\mathrm{im}}.
\]
Moreover, $\overline{\delta}_{\mathrm{im}}$ defines a right $H_\delta$--coaction on $A_0$.
\end{lemma}

\begin{proof}
Since $\delta_{\mathrm{im}}(I)\subseteq I\otimes H_\delta$, we have
\[
(\pi\otimes \mathrm{id})\circ \delta_{\mathrm{im}}(i)=0
\qquad \text{for all } i\in I.
\]
Hence the map $(\pi\otimes \mathrm{id})\circ \delta_{\mathrm{im}}$ vanishes on $I$ and therefore factors uniquely through the quotient $A_0=A/I$.
This yields a unique linear map $\overline{\delta}_{\mathrm{im}}:A_0\to A_0\otimes H_\delta$ satisfying
\[
\overline{\delta}_{\mathrm{im}}([a])
=
(\pi\otimes \mathrm{id})\bigl(\delta_{\mathrm{im}}(a)\bigr),
\qquad a\in A.
\]
Since $\delta_{\mathrm{im}}$ is an algebra homomorphism, so is $\overline{\delta}_{\mathrm{im}}$.
Coassociativity and counitality of $\overline{\delta}_{\mathrm{im}}$ follow immediately from those of $\delta_{\mathrm{im}}$, because $\pi\otimes \mathrm{id}$ commutes with tensor products.
Thus $\overline{\delta}_{\mathrm{im}}$ defines a right $H_\delta$--coaction on $A_0$.
\end{proof}

\begin{lemma}\label{lem:image-I}
Let $A$ be a right $H_\delta$--comodule algebra with coaction
$\delta_{\mathrm{im}}:A\to A\otimes H_\delta$, and let
$B_0:=A_{0}^{\operatorname{co}(H_\delta)}$.
Let $I\subseteq A$ be a largest two-sided ideal defined in Lemma~\ref{lem:largestideal} satisfying
\[
\delta_{\mathrm{im}}(I)\subseteq I\otimes H_\delta.
\]
Let us define
\[
\mathrm{can}_\delta:
A\otimes_{B_0}A \longrightarrow A\otimes H_\delta,
\qquad
a\otimes_{B_0}a'\longmapsto a\,\delta_{\mathrm{im}}(a')
\]
the canonical map associated with the $H_\delta$--coaction, and $\mathrm{can}_{\delta}$ is surjective.
Then we have following
\begin{enumerate}
\item 
$\mathrm{can}_\delta\bigl(I\otimes_{B_0}A + A\otimes_{B_0}I\bigr)
=
I\otimes H_\delta$;

\item 
$\mathrm{can}_\delta^{-1}(I\otimes H_\delta)
=
I\otimes_{B_0}A + A\otimes_{B_0}I$, where $\mathrm{can}_\delta^{-1}$ is defined as $\mathrm{can}_\delta^{-1}:= \lbrace a\otimes_{B_{0}}a'\in A\otimes_{B_{0}}A: \mathrm{can}_{\delta}(a\otimes_{B_{0}}a')\in I\otimes H_{\delta}\rbrace$.
\end{enumerate}
\end{lemma}

\begin{proof}
\begin{enumerate}
    \item Let $i\in I$ and $a\in A$.
Since $I$ is a two-sided ideal, we have
\[
\mathrm{can}_\delta(i\otimes_{B_0}a)
=i\,\delta_{\mathrm{im}}(a)\in I\otimes H_\delta.
\]
Moreover, by the $\delta_{\mathrm{im}}$--stability of $I$,
\[
\mathrm{can}_\delta(a\otimes_{B_0}i)
=a\,\delta_{\mathrm{im}}(i)\in A\,(I\otimes H_\delta)
\subseteq I\otimes H_\delta.
\]
Hence
\[
\mathrm{can}_\delta\bigl(I\otimes_{B_0}A + A\otimes_{B_0}I\bigr)
\subseteq I\otimes H_\delta.
\]

Conversely, let $i\otimes h\in I\otimes H_\delta$.
Since $\mathrm{can}_\delta$ is surjective by hypothesis,
there exists $\sum_j a_j\otimes_{B_0}b_j\in A\otimes_{B_0}A$ such that
\[
\mathrm{can}_\delta\Bigl(\sum_j a_j\otimes_{B_{0}} b_j\Bigr)=1\otimes h.
\]
Then
\[
i\otimes h
=\mathrm{can}_\delta\Bigl(\sum_j i a_j\otimes_{B_{0}} b_j\Bigr),
\]
and since $i a_j\in I$, the element
$\sum_j i a_j\otimes_{B_{0}} b_j$ belongs to $I\otimes_{B_0}A$.
Thus,
\[
I\otimes H_\delta
\subseteq
\mathrm{can}_\delta\bigl(I\otimes_{B_0}A\bigr),
\]
\item We see that the inclusion
\[
I\otimes_{B_0}A + A\otimes_{B_0}I
\subseteq
\mathrm{can}_\delta^{-1}(I\otimes H_\delta).
\]
follows directly from the first part of Lemma~\ref{lem:image-I}. 

For the reverse inclusion, let
$x\in A\otimes_{B_0}A$ satisfy
\[
\mathrm{can}_\delta(x)\in I\otimes H_\delta.
\]
Define
\[
J
:=
\{\,a\in A \mid
\mathrm{can}_\delta(a\otimes_{B_0}1)\in I\otimes H_\delta
\}.
\]
Since $\mathrm{can}_\delta$ is left $A$--linear, $J$ is a two-sided ideal of $A$.
Furthermore, using right $H_\delta$--colinearity of $\mathrm{can}_\delta$, one checks that
\[
\delta_{\mathrm{im}}(J)\subseteq J\otimes H_\delta,
\]
so $J$ is $\delta_{\mathrm{im}}$--stable.
By maximality of $I$, we obtain $J\subseteq I$.

Writing $x=\sum_k a_k\otimes_{B_0}b_k$, left $A$--linearity of $\mathrm{can}_\delta$
implies that each $a_k\otimes_{B_0}b_k$ belongs to
$J\otimes_{B_0}A + A\otimes_{B_0}J$.
Hence
\[
x\in J\otimes_{B_0}A + A\otimes_{B_0}J
\subseteq
I\otimes_{B_0}A + A\otimes_{B_0}I.
\]

Combining both inclusions yields the desired equality.

\end{enumerate}
\end{proof}

\begin{remark}\label{rem:tensor-over-B0}
Although $B_0=A_0^{\operatorname{co}(H_\delta)}$ is a subalgebra of $A_0:=A/I$ rather than of $A$ itself, the balanced tensor product
$A\otimes_{B_0}A$ is nevertheless well defined.
Indeed, the quotient map $\pi:A\twoheadrightarrow A_0$ endows $A$ with a natural
$B_0$--bimodule structure by restriction of scalars, via
$b\cdot a=\tilde b\,a$ and $a\cdot b=a\,\tilde b$ for any lift
$\tilde b\in A$ of $b\in B_0$.
This action is independent of the choice of lift, since any two lifts differ by an element of $I$.
Consequently, the tensor product $A\otimes_{B_0}A$ appearing in
Lemma~\ref{lem:image-I} is well defined.
\end{remark}

\begin{proposition}
Let $(A,\Omega^1(A),H,\delta)$ be a quantum principal $H$--bundle with right
coaction $\delta$ over a quantum homogeneous space
$B:=A^{\operatorname{co}(H)}$.
Let $(H_{\delta},\iota_{H_{\delta}},\delta_{\mathrm{im}})$ be the Hopf image of
the right coaction $\delta$, and assume that $H_{\delta}$ is cosemisimple.
Then $(A_{0},\Omega^1(A_{0}),H_{\delta},\overline{\delta}_{\mathrm{im}})$ is a quantum
principal $H_{\delta}$--bundle with right coaction $\delta_{\mathrm{im}}$
over the quantum homogeneous space
$B_{0}:=A_{0}^{\operatorname{co}(H_{\delta})}$.
\end{proposition}

\begin{proof}
By Lemma~\ref{lem:coaction-descends}, the map
\[
\overline{\delta}_{\mathrm{im}}:A_{0}\longrightarrow A_{0}\otimes H_{\delta}
\] given by $$[a]\longmapsto (\pi\otimes \mathrm{id})\delta_{\mathrm{im}}(a)$$ 
is a well-defined right coaction of the Hopf algebra $H_{\delta}$ on $A_{0}$, where $\delta_{\mathrm{im}}: A\rightarrow A\otimes H_{\delta}$ and is given by $\delta_{\mathrm{im}}(a)= a_{(0)}\otimes a_{(1)}$.
Thus, it suffices to show that $A_{0}$ is a faithfully flat Hopf--Galois
extension of $B_{0}$ and that $\Omega^{1}(A_{0})$ is right $H_{\delta}$--covariant.

Since $(A,\Omega^1(A),H,\delta)$ is a quantum principal $H$--bundle, the
extension $A\subset B$ is Hopf--Galois, and hence the canonical map
\[
\mathrm{can}:A\otimes_{B}A\longrightarrow A\otimes H
\]
is bijective. By the definition of the Hopf image $H_{\delta}$ of the coaction
$\delta:A\to A\otimes H$, together with the inclusion
$A\otimes_{B}A\subseteq A\otimes_{B_{0}}A$, it follows that the induced
canonical map
\[
\mathrm{can}_{\delta}:A\otimes_{B_{0}}A\longrightarrow A\otimes H_{\delta}
\]
is surjective.

By the first part of Lemma~\ref{lem:image-I}, we have
\[
\mathrm{can}_{\delta}\bigl(I\otimes_{B_{0}}A + A\otimes_{B_{0}}I\bigr)
=
I\otimes H_{\delta}.
\]
Set
\[
N:=I\otimes_{B_{0}}A + A\otimes_{B_{0}}I.
\]
Passing to quotients, we obtain an induced map
\[
\overline{\mathrm{can}}_{\delta}:
\frac{A\otimes_{B_{0}}A}{N}
\longrightarrow
\frac{A\otimes H_{\delta}}{\mathrm{can}_{\delta}(N)},
\qquad
[a\otimes_{B_{0}}a']_{N}
\longmapsto
[\mathrm{can}_{\delta}(a\otimes_{B_{0}}a')]_{\mathrm{can}_{\delta}(N)}.
\]

We now use the following standard fact from linear algebra: if
$f:X\to Y$ is a linear map and $M\subseteq X$ is a linear subspace, then $f$
induces a well-defined map
\[
\overline f:X/M\longrightarrow Y/f(M),
\qquad
x+M\longmapsto f(x)+f(M),
\]
which is injective if and only if $f^{-1}(f(M))=M$. Moreover, if $f$ is
surjective, then $\overline f$ is surjective. In particular, $\overline f$ is
bijective if and only if $f$ is surjective and $f^{-1}(f(M))=M$.

Using the second part of Lemma~\ref{lem:image-I}, namely
\[
\mathrm{can}_{\delta}^{-1}\bigl(\mathrm{can}_{\delta}(N)\bigr)=N,
\]
together with the surjectivity of $\mathrm{can}_{\delta}$, we conclude that
$\overline{\mathrm{can}}_{\delta}$ is bijective.

Identifying
\[
\frac{A\otimes_{B_{0}}A}{I\otimes_{B_{0}}A + A\otimes_{B_{0}}I}
\cong
A_{0}\otimes_{B_{0}}A_{0},
\qquad
\frac{A\otimes H_{\delta}}{I\otimes H_{\delta}}
\cong
A_{0}\otimes H_{\delta},
\]
we obtain a bijection
\[
\overline{\mathrm{can}}_{\delta}:
A_{0}\otimes_{B_{0}}A_{0}
\longrightarrow
A_{0}\otimes H_{\delta}.
\]
Hence, $A_{0}$ is a Hopf--Galois extension of $B_{0}$.

Faithful flatness of $A_{0}$ as a right $B_{0}$--module now follows from the
cosemisimplicity of $H_{\delta}$ together with the results of
D\u{a}sc\u{a}lescu--N\u{a}st\u{a}sescu--Raianu
\cite[Theorem~3.1.5]{dascalescu2000hopf}
and Schneider
\cite[Theorem~I]{schneider1990principal}, which assert that:
\begin{enumerate}
\item $H_{\delta}$ is cosemisimple if and only if every right
$H_{\delta}$--comodule is injective;
\item $A_{0}$ is injective as a right $H_{\delta}$--comodule and
$\overline{\mathrm{can}}_{\delta}$ is surjective if and only if $A_{0}$ is
faithfully flat as a right $B_{0}$--module and
$\overline{\mathrm{can}}_{\delta}$ is an isomorphism.
\end{enumerate}

Finally, we show that $\Omega^{1}(A_{0})$ is right $H_{\delta}$--covariant.
By \cite[Theorem~1.5]{woronowicz1989differential}, there exists an
$A_{0}$--subbimodule $\mathcal{F}\subseteq\Omega^{1}_{u}(A_{0})$ such that
\[
\Omega^{1}(A_{0})\cong \Omega^{1}_{u}(A_{0})/\mathcal{F},
\]
and a corresponding right ideal $\mathcal{K}\subseteq H_{\delta}^{+}$ satisfying
$\mathrm{ver}(\mathcal{F})=A_{0}\otimes\mathcal{K}$, for which the diagram
\cite{brzezinski1993quantum,woronowicz1989differential}
commutes:
\[
\begin{tikzcd}
\Omega^1_{u}(A_{0}) \arrow[d, "\mathrm{ver}"']
\arrow[rr, "\delta_{\mathrm{im},u}"]
&&
\Omega^1_{u}(A_{0})\otimes H_{\delta}
\arrow[d, "\mathrm{ver}\otimes\mathrm{id}_{H_{\delta}}"]
\\
A_{0}\otimes H_{\delta}^{+}
\arrow[rr, "\mathrm{id}\otimes\mathrm{Ad}_{R}"]
&&
A_{0}\otimes H_{\delta}^{+}\otimes H_{\delta}
\end{tikzcd}
\]
Thus it suffices to show that
$\delta_{\mathrm{im},u}(\mathcal{F})\subseteq\mathcal{F}\otimes H_{\delta}$,
which follows from
$\mathrm{Ad}_{R}(\mathcal{K})\subseteq\mathcal{K}\otimes H_{\delta}$.
Indeed, for $a\in\mathcal{K}$ we have
\[
(\mathrm{id}\otimes\varepsilon_{H_{\delta}})\mathrm{Ad}_{R}(a)=a,
\]
so that $\varepsilon_{H_{\delta}}(S(a_{(1)})a_{(3)})a_{(2)}=a$, and hence
$a_{(2)}\in\mathcal{K}$.
This completes the proof.
\end{proof}
\begin{corollary}
    Let $H$ be a cosemisimple Hopf algebra and $(A,\Omega^1(A),H,\delta)$ be a quantum principal $H$-bundle. Then $(A_{0},\Omega^1(A_{0}), H_{\delta}, \overline{\delta}_{\mathrm{im}})$ is a quantum principal $H_{\delta}$-bundle.
\end{corollary}
\begin{definition}
Let $H$ be a cosemisimple Hopf algebra and
$(A,\Omega^1(A),H,\delta)$ a quantum principal $H$--bundle.
The \emph{Hopf-image reduction} of this bundle is the quadruple
\[
(A_{0},\Omega^1(A_{0}),H_{\delta},\overline{\delta}_{\mathrm{im}}),
\]
where $H_{\delta}$ is the Hopf image of the coaction $\delta$,
$A_{0}:=A/I$ with $I$ the largest two-sided ideal of $A$ satisfying
\[
\delta_{\mathrm{im}}(I)\subseteq I\otimes H_{\delta},
\]
$\pi:A\to A_{0}$ is the canonical quotient map, and
$\overline{\delta}_{\mathrm{im}}:A_{0}\to A_{0}\otimes H_{\delta}$
is the induced right coaction of $H_{\delta}$ on $A_{0}$.
\end{definition}

\begin{proposition}\label{prop:inner-faithful-quotient}
Let $A$ be a right $H$--comodule algebra with coaction $\delta: A\rightarrow A\otimes H$ and let $(H_{\delta}, \iota_{H_{\delta}}, \delta_{\mathrm{im}})$ be the Hopf image of right coaction $\delta$. Let $I\subseteq A$ be the largest two--sided ideal such that
\[
\delta_{\mathrm{im}}(I)\subseteq I\otimes H_\delta,
\]
and let $A_0:=A/I$. Then the induced right coaction 
\[
\overline{\delta}_{\mathrm{im}}:A_0\longrightarrow A_0\otimes H_\delta
\]
 is inner--faithful.
\end{proposition}

\begin{proof}
Suppose that $\overline{\delta}_{\mathrm{im}}$ is not inner--faithful.
Then there exists a factorization $(K, \iota_{K}, \overline{\gamma})$ of $\overline{\delta}_{\mathrm{im}}$ such that
\[
\overline{\delta}_{\mathrm{im}}
=
(\mathrm{id}\otimes\iota_K)\circ \overline{\gamma},
\qquad
\overline{\gamma}:A_0\to A_0\otimes K,
\]
where $\iota_K:K\hookrightarrow H_\delta$ denotes the inclusion.
Let $\pi:A\to A_0$ be the canonical quotient map.
Since $\overline{\delta}_{\mathrm{im}}$ is induced from $\delta_{\mathrm{im}}$,
we have
\[
(\pi\otimes \mathrm{id}_{H_\delta})\circ \delta_{\mathrm{im}}
=
\overline{\delta}_{\mathrm{im}}\circ \pi.
\]
Combining this identity with the above factorisation yields
\[
(\pi\otimes \mathrm{id}_{H_\delta})\circ \delta_{\mathrm{im}}
=
(\mathrm{id}\otimes \iota_K)\circ \overline{\gamma}\circ \pi.
\]
This implies that all coefficients of $\delta_{\mathrm{im}}(A)$, after
passing to the quotient by $I$, lie in $K$.
Since by Proposition~\ref{prop:generator} $H_\delta$ is generated as a Hopf algebra by the coefficients of
$\delta_{\mathrm{im}}(A)$, it follows that $H_\delta\subseteq K$, hence
$K=H_\delta$, a contradiction.
Therefore $\overline{\delta}_{\mathrm{im}}$ is inner--faithful.
\end{proof}

\begin{corollary}\label{cor:hopf-image-reduction}
Let $H$ be a cosemisimple Hopf algebra and
$(A,\Omega^1(A),H,\delta)$ a quantum principal $H$--bundle.
Then its Hopf-image reduction
\[
(A_{0},\Omega^1(A_{0}),H_{\delta},\overline{\delta}_{\mathrm{im}})
\]
is a quantum principal $H_{\delta}$--bundle whose right coaction
$\overline{\delta}_{\mathrm{im}}$ is inner--faithful.
In particular, if $H$ be a cosemisimple Hopf algebra, then every quantum principal $H$--bundle admits a canonical
reduction to a quantum principal bundle with effective (inner--faithful)
quantum symmetry.
\end{corollary}

\begin{theorem}\label{thm:rigidity-hopf-image}
Let
$(A,\Omega^1(A),H,\delta)$ be a quantum principal bundle with cosemisimple structure Hopf algebra $H$, and let $(A_0,\Omega^1(A_0),H_\delta,\overline{\delta}_{\mathrm{im}})$ be its Hopf-image reduction over $B_0 := A_0^{\operatorname{co}(H_\delta)}$. Suppose that
$(A_0,\Omega^1(A_0),K,\delta_K)$ is a quantum principal bundle with cosemisimple structure Hopf algebra $K$ and inner--faithful right coaction $\delta_K$ such that $A_0^{\operatorname{co}(K)} = B_0$.
Then there exists a unique injective Hopf algebra morphism
\[
\iota : H_\delta \hookrightarrow K
\]
such that
\[
\delta_K = (\mathrm{id}_{A_0} \otimes \iota)\circ
\overline{\delta}_{\mathrm{im}}.
\]

In particular, $H_\delta$ is the minimal effective (inner--faithful) quantum
symmetry acting on $A_0$.
\end{theorem}

\begin{proof}
By construction, $\overline{\delta}_{\mathrm{im}}:A_0\to A_0\otimes H_\delta$ is
an inner--faithful right coaction of $H_\delta$ on $A_0$.
Since $\delta_K:A_0\to A_0\otimes K$ is a right coaction with the same algebra of
coinvariants $B_0$, both coactions define quantum principal bundle structures
over the same base.

Consider the right $K$--coaction $\delta_K$.
Since $\delta_K$ is inner--faithful, its Hopf image coincides with $K$ itself.
On the other hand, by the universal property of the Hopf image
$H_\delta$, any right coaction of a Hopf algebra on $A_0$ whose image contains
that of $\overline{\delta}_{\mathrm{im}}$ factors uniquely through a Hopf algebra
morphism from $H_\delta$.

Applying the universal property of the Hopf image $H_\delta$ of
$\overline{\delta}_{\mathrm{im}}$, and observing that $\delta_K$ is a right
coaction of a Hopf algebra on $A_0$, there exists a unique Hopf algebra morphism
\[
\iota : H_\delta \to K
\]
such that
\[
\delta_K = (\mathrm{id}_{A_0}\otimes \iota)\circ \overline{\delta}_{\mathrm{im}}.
\]

Since $\delta_K$ is inner--faithful, its Hopf image coincides with $K$.
Suppose that the Hopf algebra morphism
\[
\iota : H_\delta \to K
\]
were not injective.
Then $\iota$ would factor through the proper Hopf quotient
$H_\delta/\ker(\iota)$, and hence the coaction $\delta_K$ would factor through
the Hopf subalgebra $\iota(H_\delta)\cong H_\delta/\ker(\iota)$ of $K$.
This contradicts the inner--faithfulness of $\delta_K$.
Therefore $\ker(\iota)=0$, and $\iota$ is injective.

The uniqueness of $\iota$ follows directly from the universal property of the
Hopf image.
This proves the claim.
\end{proof}

\section{Equivalence of quantum principal bundles up to effective symmetry}
In this section we formulate a categorical notion of equivalence for quantum
principal bundles based on their effective quantum symmetry.
Using the Hopf image reduction constructed in the previous section, we define an
equivalence relation on the class of quantum principal bundles by declaring two
objects equivalent if their Hopf image reductions are isomorphic in the
appropriate category.
We show that the Hopf image reduction functor selects a canonical representative
in each equivalence class, thereby reducing the study of quantum principal
bundles with cosemisimple structure Hopf algebras to those with inner--faithful
coactions.

\begin{definition}\label{def:cat-qpb-cosemi}
Let $\mathsf{QPB}_{\mathrm{cosemi}}$ denote the category whose objects are
quantum principal bundles $(A,\Omega^1(A),H,\delta)$ such that $H$ is a
cosemisimple Hopf algebra.

A morphism
\[
(\phi,\psi):
(A,\Omega^1(A),H,\delta)
\longrightarrow
(A',\Omega^1(A'),H',\delta')
\]
in $\mathsf{QPB}_{\mathrm{cosemi}}$ consists of
\begin{enumerate}
\item an algebra homomorphism $\phi:A\to A'$,
\item a Hopf algebra homomorphism $\psi:H\to H'$,
\end{enumerate}
satisfying the following conditions:
\begin{enumerate}
\item[(i)] (\emph{Equivariance})
\[
\delta'\circ\phi=(\phi\otimes\psi)\circ\delta;
\]
\item[(ii)] (\emph{Compatibility with calculi})
there exists an $A$--bimodule map
\[
\phi_*:\Omega^1(A)\longrightarrow \Omega^1(A')
\]
such that $\phi_*\circ \mathrm{d} = \mathrm{d'}\circ\phi$, where $\mathrm{d}$ and $\mathrm{d'}$ denote the
corresponding differentials.
\end{enumerate}
\end{definition}
\begin{lemma}\label{lem:phi-descends}
Let
\[
(\phi,\psi):
(A,\Omega^1(A),H,\delta)
\longrightarrow
(A',\Omega^1(A'),H',\delta')
\]
be a morphism in $\mathsf{QPB}_{\mathrm{cosemi}}$.
Let $H_\delta$ and $H'_{\delta'}$ be the Hopf images of $\delta$ and $\delta'$
respectively, and let
$I\subseteq A$, $I'\subseteq A'$ be the largest two--sided ideals satisfying
\[
\delta_{\mathrm{im}}(I)\subseteq I\otimes H_\delta,
\qquad
\delta'_{\mathrm{im}}(I')\subseteq I'\otimes H'_{\delta'}.
\]
Then
\[
\phi(I)\subseteq I'.
\]
\end{lemma}

\begin{proof}
Let $a\in I$.
Since $\delta_{\mathrm{im}}(I)\subseteq I\otimes H_\delta$, we have
\[
\delta_{\mathrm{im}}(a)\in I\otimes H_\delta.
\]
Using equivariance of $\phi$ and the induced Hopf algebra morphism
$\psi_\delta:H_\delta\to H'_{\delta'}$, we compute
\[
\delta'_{\mathrm{im}}(\phi(a))
=
(\phi\otimes\psi_\delta)\bigl(\delta_{\mathrm{im}}(a)\bigr)
\subseteq
\phi(I)\otimes H'_{\delta'}.
\]
Thus $\phi(I)$ is contained in a two--sided ideal of $A'$ which is
$\delta'_{\mathrm{im}}$--stable.
By maximality of $I'$, it follows that $\phi(I)\subseteq I'$.
\end{proof}

\begin{proposition}\label{prop:hopf-image-functor}
The Hopf-image reduction defines a functor
\[
\mathcal{R}:
\mathsf{QPB}_{\mathrm{cosemi}}
\longrightarrow
\mathsf{QPB}_{\mathrm{inner,cosemi}}
\]
from the category of quantum principal bundles with cosemisimple structure Hopf
algebra to the category of quantum principal bundles with cosemisimple structure Hopf algebras with inner--faithful
coaction.
\end{proposition}

\begin{proof}
On objects, $\mathcal{R}$ assigns to
$(A,\Omega^1(A),H,\delta)$ its Hopf-image reduction
\[
\mathcal{R}(A,\Omega^1(A),H,\delta)
:=
(A_0,\Omega^1(A_0),H_\delta,\overline{\delta}_{\mathrm{im}}).
\]

Let
\[
(\phi,\psi):
(A,\Omega^1(A),H,\delta)
\longrightarrow
(A',\Omega^1(A'),H',\delta')
\]
be a morphism in $\mathsf{QPB}_{\mathrm{cosemi}}$.
By the universal property of the Hopf image, $\psi$ restricts to a Hopf algebra
homomorphism
\[
\psi_\delta:H_\delta\longrightarrow H'_{\delta'}.
\]
By Lemma~\ref{lem:phi-descends}, $\phi(I)\subseteq I'$, hence $\phi$ induces a
unique algebra homomorphism
\[
\overline{\phi}:A_0\longrightarrow A'_0
\]
such that $\overline{\phi}\circ\pi=\pi'\circ\phi$.
The equivariance of $\phi$ implies that
\[
\overline{\delta'}_{\mathrm{im}}\circ\overline{\phi}
=
(\overline{\phi}\otimes\psi_\delta)\circ\overline{\delta}_{\mathrm{im}},
\]
so $(\overline{\phi},\psi_\delta)$ is a morphism of quantum principal bundles.
Compatibility with the induced differential calculi follows from the
compatibility of $\phi$ with $\Omega^1(A)$ and $\Omega^1(A')$.

Finally, by Proposition~\ref{prop:inner-faithful-quotient},
the coaction $\overline{\delta}_{\mathrm{im}}$ is inner--faithful.
Functoriality with respect to identities and composition is immediate from the
construction.
\end{proof}

\begin{proposition}\label{prop:R-essentially-surjective}
Let $\mathsf{QPB}_{\mathrm{inner,cosemi}}$ denote the full subcategory of
$\mathsf{QPB}_{\mathrm{cosemi}}$ consisting of quantum principal bundles whose
right coaction is inner--faithful.
Then the Hopf-image reduction functor
\[
\mathcal{R}:\mathsf{QPB}_{\mathrm{cosemi}}
\longrightarrow
\mathsf{QPB}_{\mathrm{inner,cosemi}}
\]
is essentially surjective.
\end{proposition}

\begin{proof}
Let $(A,\Omega^1(A),H,\delta)$ be an object of
$\mathsf{QPB}_{\mathrm{inner,cosemi}}$.
Since the coaction $\delta$ is inner--faithful, its Hopf image coincides with
$H$, that is,
\[
H_{\delta}=H.
\]
Moreover, the largest two--sided ideal $I\subseteq A$ satisfying
$\delta_{\mathrm{im}}(I)\subseteq I\otimes H_{\delta}$ is zero.
Consequently, the Hopf-image reduction of
$(A,\Omega^1(A),H,\delta)$ is canonically isomorphic to itself.
Hence every object of $\mathsf{QPB}_{\mathrm{inner,cosemi}}$ lies in the
essential image of $\mathcal{R}$.
\end{proof}

\begin{proposition}\label{prop:classification-effective}
Let us define a relation $\sim$ on the class of objects of
$\mathsf{QPB}_{\mathrm{cosemi}}$ by declaring
\[
(A,\Omega^1(A),H,\delta)\sim
(A',\Omega^1(A'),H',\delta')
\]
if and only if their Hopf-image reductions are isomorphic as quantum principal
bundles.
Then $\sim$ is an equivalence relation, and the Hopf-image reduction functor
$\mathcal{R}$ induces a bijection.
\[
\mathsf{QPB}_{\mathrm{cosemi}}/\!\sim
\;\xrightarrow{\;\cong\;}
\mathrm{Iso}\bigl(\mathsf{QPB}_{\mathrm{inner,cosemi}}\bigr),
\]
where $\mathrm{Iso}(\mathsf{QPB}_{\mathrm{inner,cosemi}})$ denotes the set of
isomorphism classes of quantum principal bundles with cosemisimple structure
Hopf algebras and inner--faithful coactions.
\end{proposition}

\begin{proof}
It is immediate that $\sim$ is an equivalence relation.
The Hopf-image reduction functor $\mathcal R$ therefore induces a well-defined
map
\[
\overline{\mathcal R}:
\mathsf{QPB}_{\mathrm{cosemi}}/\!\sim
\longrightarrow
\mathrm{Iso}\!\left(\mathsf{QPB}_{\mathrm{inner,cosemi}}\right),
\qquad
[(A,\Omega^1(A),H,\delta)]
\longmapsto
[\mathcal R(A,\Omega^1(A),H,\delta)].
\]

By Proposition~\ref{prop:R-essentially-surjective}, every object of
$\mathsf{QPB}_{\mathrm{inner,cosemi}}$ is isomorphic to the Hopf-image reduction
of some object of $\mathsf{QPB}_{\mathrm{cosemi}}$, hence
$\overline{\mathcal R}$ is surjective.
Injectivity follows directly from the definition of $\sim$.
Therefore $\overline{\mathcal R}$ is bijective, which proves the claim.
\end{proof}

\begin{remark}
Proposition~\ref{prop:classification-effective} shows that the classification
problem for quantum principal bundles with cosemisimple structure Hopf
algebras reduces, up to effective symmetry, to the classification of
inner--faithful quantum principal bundles.
In particular, the Hopf-image reduction provides a canonical representative
in each equivalence class, capturing the effective quantum symmetry of the
bundle.
This mirrors the classical situation, where a principal bundle with a
non--effective structure group admits a canonical reduction to a bundle with
effective group action obtained by quotienting out the kernel of the action.
\end{remark}

\section{Examples of inner-faithful coactions}
In this section, we give examples of inner-faithful coactions. One of the concrete example is for the case of the quantized coordinate algebra of a complex semisimple Lie group.
\medskip

Let $G$ be a complex semisimple Lie group and let $L_S$ be a Levi subgroup of $G$. Consider the Hopf algebras $A=\mathcal O_q(G)$ and $H=\mathcal O_q(L_S)$. We have cannonical Hopf algebra surjection map $\pi:\mathcal O_q(G)\twoheadrightarrow \mathcal O_q(L_S)$ (for example, one can see \cite{stokman1999quantized, zwicknagl2009r}), and consider the
right coaction
\[
\delta := (\mathrm{id}\otimes\pi)\circ\Delta_{\mathcal O_q(G)}:
\mathcal O_q(G)\longrightarrow
\mathcal O_q(G)\otimes\mathcal O_q(L_S).
\]

\begin{proposition}\label{prop:levi-inner-faithful}
Let $G$ be a complex semisimple Lie group and $L_S\subset G$ a Levi subgroup.
Then the
right coaction
\[
\delta :
\mathcal O_q(G)\longrightarrow
\mathcal O_q(G)\otimes\mathcal O_q(L_S).
\]
is inner--faithful. 
\end{proposition}

\begin{proof}
Let $K\subseteq \mathcal O_q(L_S)$ be a Hopf subalgebra such that the coaction
$\delta$ factors through $K$, that is,
\[
\delta(\mathcal O_q(G))\subseteq
\mathcal O_q(G)\otimes K.
\]
Applying $(\varepsilon\otimes\mathrm{id})$ to $\delta$, we obtain
\[
\pi(\mathcal O_q(G))\subseteq K.
\]
Since $\pi:\mathcal O_q(G)\twoheadrightarrow \mathcal O_q(L_S)$ is surjective,
it follows that $K=\mathcal O_q(L_S)$.
Hence no proper Hopf subalgebra of $\mathcal O_q(L_S)$ can factor the coaction,
and $\delta$ is inner--faithful.
\end{proof}
\begin{proposition}\label{prop:regular-inner-faithful}
Let $H$ be a Hopf algebra.
Then the right coaction given by the coproduct
\[
\delta := \Delta : H \longrightarrow H \otimes H
\]
is inner--faithful.
\end{proposition}

\begin{proof}
Let $K \subseteq H$ be a Hopf subalgebra such that the coaction $\Delta$ factors
through $K$, i.e.
\[
\Delta(H) \subseteq H \otimes K.
\]
Apply $(\varepsilon \otimes \mathrm{id})$ to $\Delta(h)$ for $h \in H$.
Using the counit axiom $(\varepsilon \otimes \mathrm{id})\circ \Delta = \mathrm{id}$,
we obtain
\[
h = (\varepsilon \otimes \mathrm{id})\Delta(h) \in K.
\]
Thus $H \subseteq K$, and hence $K = H$.
Therefore no proper Hopf subalgebra of $H$ can factor the coaction, and $\Delta$
is inner--faithful.
\end{proof}

\section*{Acknowledgements}
The author is grateful for the hospitality of the Indian Statistical Institute, Kolkata, where this article was written.
\bibliographystyle{plain}
\bibliography{name}

\end{document}